\documentclass[10pt]{amsart}

\usepackage{amssymb,amsmath,amsthm}
\usepackage[all]{xy}
\usepackage{eucal}

\usepackage{color}

\theoremstyle{plain}
\newtheorem{thm}[subsection]{Theorem}

\newtheorem{prop}[subsection]{Proposition}

\theoremstyle{definition}
\newtheorem{defn}[subsection]{Definition}

\theoremstyle{remark}
\newtheorem{rem}[subsection]{Remark}

\makeatletter
\let\c@equation\c@subsection

\makeatother

\newcommand{\ZZ}{{ \mathbb{Z} }}
\newcommand{\NN}{{ \mathbb{N} }}

\newcommand{\capC}{{ \mathcal{C} }}

\newcommand{\capP}{{ \mathcal{P} }}

\newcommand{\capX}{{ \mathcal{X} }}
\newcommand{\capY}{{ \mathcal{Y} }}

\newcommand{\id}{{ \mathrm{id} }}

\newcommand{\Space}{{ \mathsf{S} }}

\newcommand{\Ho}{{ \mathsf{Ho} }}

\newcommand{\sSet}{{ \mathsf{sSet} }}

\newcommand{\M}{{ \mathsf{M} }}

\newcommand{\Set}{{ \mathsf{Set} }}

\newcommand{\coAlg}{{ \mathsf{coAlg} }}
\newcommand{\res}{{ \mathsf{res} }}
\newcommand{\BK}{{ \mathsf{BK} }}
\newcommand{\CGHaus}{{ \mathsf{CGHaus} }}

\newcommand{\coAlgK}{{ \coAlg_\K }}
\newcommand{\coAlgKt}{{ \coAlg_{\tilde{\K}} }}

\newcommand{\Loop}{{ \Omega }}
\newcommand{\Loopt}{{ \tilde{\Omega} }}
\newcommand{\Susp}{{ \Sigma }}
\newcommand{\K}{{ \mathsf{K} }}
\newcommand{\Kt}{{ \tilde{\mathsf{K}} }}

\newcommand{\Smash}{{ \,\wedge\, }}
\newcommand{\tensor}{{ \otimes }}

\newcommand{\tensordot}{{ \dot{\tensor} }}

\newcommand{\wequiv}{{ \ \simeq \ }}

\newcommand{\Iso}{{  \ \cong \ }}

\newcommand{\rarrow}{{ \rightarrow }}

\newcommand{\function}[3]{{ {#1}\colon\thinspace{#2}\rarrow{#3} }}
\newcommand{\functionlong}[3]{{ {#1}\colon\thinspace{#2}\longrightarrow{#3} }}

\DeclareMathOperator*{\holim}{holim}
\DeclareMathOperator{\Hombold}{\mathbf{Hom}}
\DeclareMathOperator{\hombold}{\mathbf{hom}}

\DeclareMathOperator{\Map}{Map}

\DeclareMathOperator{\Cobar}{Cobar}

\DeclareMathOperator{\Tot}{Tot}

\DeclareMathOperator{\hofib}{hofib}
\DeclareMathOperator{\iter}{iterated}

\title[Iterated suspension spaces]{Iterated suspension spaces and higher {F}reudenthal suspension}

\author{Jacobson R. Blomquist}
\author{John E. Harper}

\address{Department of Mathematics, The Ohio State University, 231 West 18th Ave, Columbus, OH 43210, USA}
\email{blomquist.9@osu.edu}

\address{Department of Mathematics, The Ohio State University, Newark, 1179 University Dr, Newark, OH 43055, USA}
\email{harper.903@math.osu.edu}

\begin{document}

\begin{abstract}
We establish a higher Freudenthal suspension theorem and prove that the derived fundamental adjunction comparing spaces with coalgebra spaces over the homotopical iterated suspension-loop comonad, via iterated suspension, can be turned into an equivalence of homotopy theories by replacing spaces with the full subcategory of $1$-connected spaces. This resolves in the affirmative a conjecture of Lawson on iterated suspension spaces; that homotopical descent for iterated suspension is satisfied on objects and morphisms---the corresponding iterated desuspension space can be built as the homotopy limit of a cosimplicial cobar construction encoding the homotopical coalgebraic structure. It also provides a homotopical recognition principle for iterated suspension spaces. In a nutshell, we show that the iterated loop-suspension completion map studied by Bousfield participates in a derived equivalence between spaces and coalgebra spaces over the associated homotopical comonad, after restricting to $1$-connected spaces. 
\end{abstract}

\maketitle

\section{Introduction}

In this paper we work simplicially so that ``space'' means ``simplicial set'' unless otherwise noted; see, for instance, Bousfield-Kan \cite{Bousfield_Kan}, Dundas-Goodwillie-McCarthy \cite{Dundas_Goodwillie_McCarthy}, Goerss-Jardine \cite{Goerss_Jardine}, and Hovey \cite{Hovey}; an introductory treatment is provided in Dwyer-Henn \cite{Dwyer_Henn}.

\subsection{The spaces-level Freudenthal suspension map}

If $X$ is a pointed space and $r\geq 1$, the Freudenthal suspension map has the form
\begin{align}
\label{eq:freudenthal_suspension_map_introduction}
  \pi_*(X)\rarrow \pi_{*+r}(\Sigma^r X)
\end{align}
This map comes from a spaces-level Freudenthal suspension map of the form
\begin{align}
\label{eq:freudenthal_suspension_map_spaces_level_introduction}
  X\rightarrow\tilde{\Loop}^r\Susp^r(X)
\end{align}
and can be thought of as an analog of the spaces-level Hurewicz map $X\rarrow\tilde{\ZZ} X$ that underlies the work in Bousfield-Kan \cite{Bousfield_Kan}, Dundas \cite{Dundas_relative_K_theory}, and Dundas-Goodwillie-McCarthy \cite{Dundas_Goodwillie_McCarthy}, and subsequently our work in \cite{Blomquist_Harper_integral_chains} resolving in the affirmative the integral chains problem. Applying $\pi_*$ to the map \eqref{eq:freudenthal_suspension_map_spaces_level_introduction} recovers the map \eqref{eq:freudenthal_suspension_map_introduction}; here, $\tilde{\Loop}^r$ denotes the right-derived functor of the iterated loop space $\Loop^r=\hombold_*(S^r,-)$ functor.

\subsection{Iterating the Freudenthal suspension map}

Once one has such a Freudenthal suspension map on the level of spaces, it is natural to form a cosimplicial resolution of $X$ with respect to $\Loopt^r\Sigma^r$ of the form
\begin{align}
\label{eq:iterated_loop_suspension_resolution_introduction}
\xymatrix{
  X\ar[r] &
  \tilde{\Loop}^r\Sigma^r(X)\ar@<-0.5ex>[r]\ar@<0.5ex>[r] &
  (\tilde{\Loop}^r\Sigma^r)^2(X)
  \ar@<-1.0ex>[r]\ar[r]\ar@<1.0ex>[r] &
  (\tilde{\Loop}^r\Sigma^r)^3(X)\cdots
  }
\end{align}
showing only the coface maps. The homotopical comonad $\Kt=\Susp^r\Loopt^r$, which is the derived functor of the comonad $\K=\Susp^r\Loop^r$ associated to the $(\Susp^r, \Loop^r)$ adjunction, can be thought of as encoding the spaces-level co-operations on the iterated suspension spaces. Bousfield \cite{Bousfield_cosimplicial_space} studied the cosimplicial resolution of $X$ with respect to $\Loopt^r\Sigma^r$, analogous to the homology resolutions studied in Bousfield-Kan \cite{Bousfield_Kan}, and taking the homotopy limit of the resolution \eqref{eq:iterated_loop_suspension_resolution_introduction} produces the $\Loopt^r\Susp^r$-completion map
\begin{align}
\label{eq:comparison_map_iterated_loop_suspension_completion}
  X\rarrow X^\wedge_{\Loopt^r\Susp^r}
\end{align}

\subsection{The main result}

In this paper we shall prove the following theorem, resolving in the affirmative a conjecture of Tyler Lawson \cite{Lawson_conjecture} on iterated suspension spaces; that homotopical descent for iterated suspension is satisfied on objects and morphisms---the corresponding iterated desuspension space can be built as the homotopy limit of a cosimplicial cobar construction encoding the homotopical coalgebraic structure. It also provides a homotopical recognition principle for iterated suspension spaces that can be thought of as a (dual) analog of May's \cite{May} homotopical approach to iterated loop spaces; see also Carlsson-Milgram \cite{Carlsson_Milgram} for a useful introduction to these ideas. In a nutshell, we show that the iterated loop-suspension completion map studied in Bousfield \cite{Bousfield_cosimplicial_space} participates in a derived equivalence between spaces and coalgebra spaces over the associated homotopical comonad, after restricting to $1$-connected spaces.

\begin{thm}
\label{MainTheorem}
The fundamental adjunction \eqref{eq:fundamental_adjunction_comparing_spaces_with_coAlgK}, comparing pointed spaces to coalgebra spaces over the comonad $\K=\Susp^r\Loop^r$ via iterated suspension $\Susp^r$
\begin{align}
\label{eq:fundamental_adjunction_comparing_spaces_with_coAlgK}
\xymatrix{
  \Space_*\ar@<0.5ex>[r]^-{\Susp^r} & \coAlgK\ar@<0.5ex>[l]^-{\lim_\Delta C}
}
\end{align}
 induces a derived adjunction of the form
\begin{align}
\label{eq:derived_fundamental_adjunction}
  \Map_{\coAlgKt}(\Susp^r X,Y)\wequiv
  \Map_{\Space_*}(X,\holim\nolimits_\Delta \mathfrak{C}(Y))
\end{align}
that is an equivalence of homotopy theories, after restriction to the full subcategories of $1$-connected spaces and $(1+r)$-connected $\Kt$-coalgebra spaces. More precisely:
\begin{itemize}
\item[(a)] If $Y$ is a $(1+r)$-connected $\Kt$-coalgebra space, then the derived counit map
\begin{align*}
  \Susp^r\holim\nolimits_\Delta \mathfrak{C}(Y)\xrightarrow{\wequiv} Y
\end{align*}
associated to the derived adjunction \eqref{eq:derived_fundamental_adjunction} is a weak equivalence; i.e., the iterated suspension functor $\Susp^r$ in \eqref{eq:derived_fundamental_adjunction} is homotopically essentially surjective on $(1+r)$-connected coalgebra spaces over $\Kt$.
\item[(b)] If $X'$ is a $1$-connected space, then the derived unit map
\begin{align*}
  X'\xrightarrow{\wequiv}\holim\nolimits_\Delta 
  \mathfrak{C}(\Susp^r X')
\end{align*} 
associated to the derived adjunction \eqref{eq:derived_fundamental_adjunction} is tautologically the $\Loopt^r\Susp^r$-completion map $X'\rarrow {X'}^\wedge_{\Loopt^r\Susp^r}$, and hence is a weak equivalence by \cite{Bousfield_cosimplicial_space}; in particular, the iterated suspension functor induces a weak equivalence
\begin{align}
  \label{eq:homotopically_fully_faithful_iterated_suspension}
  \Susp^r\colon\thinspace\Map^h_{\Space_*}(X,X')\wequiv\Map_\coAlgKt(\Susp^r X,\Susp^r X')
\end{align}
on mapping spaces and hence is homotopically fully faithful on $1$-connected spaces.
\end{itemize}
We denote by $\Map^h_{\Space_*}(X,X')$ the realization of the Dwyer-Kan \cite{Dwyer_Kan_function_complexes} homotopy function complex.
\end{thm}

\begin{rem}
In other words, homotopical descent for iterated suspension is satisfied on the indicated objects and morphisms.
\end{rem}

\subsection{Classification and characterization theorems}
The following are immediate consequences of our main result; compare with the analogous integral chains problem \cite{Blomquist_Harper_integral_chains} and Ching-Harper \cite{Ching_Harper_derived_Koszul_duality} on topological Quillen homology for structured ring spectra.

\begin{thm}[Classification theorem]
\label{thm:classification}
A pair of $1$-connected pointed spaces $X$ and $X'$ are weakly equivalent if and only if the iterated suspension spaces $\Susp^r X$ and $\Susp^r X'$ are weakly equivalent as derived $\Kt$-coalgebra spaces.
\end{thm}

\begin{thm}[Classification of maps theorem]
\label{thm:classification_maps}
Let $X,X'$ be pointed spaces. Assume that $X'$ is $1$-connected and fibrant.
\begin{itemize}
\item[(a)] \emph{(Existence)} Given any map $\phi$ in $[\Susp^r X,\Susp^r X']_\Kt$, there exists a map $f$ in $[X,X']$ such that $\phi=\Susp^r(f)$.
\item[(b)] \emph{(Uniqueness)} For each pair of maps $f,g$ in $[X,X']$, $f=g$ if and only if $\Susp^r(f)=\Susp^r(g)$ in the homotopy category of $\Kt$-coalgebra spaces.
\end{itemize}
\end{thm}

\begin{thm}[Characterization theorem]
\label{thm:characterization}
A $\Kt$-coalgebra space $Y$ is weakly equivalent to the iterated suspension $\Susp^r X$ of some $1$-connected space $X$, via derived $\Kt$-coalgebra maps, if and only if $Y$ is $(1+r)$-connected.
\end{thm}

\subsection{Strategy of attack and related work}
We are essentially exploiting a method of attack worked out in Ching-Harper \cite{Ching_Harper_derived_Koszul_duality} for resolving the $0$-connected case of the Francis-Gaitsgory conjecture \cite{Francis_Gaitsgory}, together with a strengthened version of that strategy for leveraging uniform cartesian-ness estimates that we developed in \cite{Blomquist_Harper_integral_chains} to resolve the integral chains problem, together with Cohn's \cite{Cohn} work showing that this strategy of attack extends to homotopy coalgebras over the associated homotopical comonad (see Blumberg-Riehl \cite{Blumberg_Riehl}). Motivated by work in Bousfield \cite{Bousfield_cosimplicial_space} on $\Loopt^r\Susp^r$-completion, together with the close-in-spirit ideas on iterated suspension spaces in Hopkins \cite{Hopkins_iterated_suspension}, and subsequently in Goerss \cite{Goerss_desuspension} and Klein-Schwanzl-Vogt \cite{Klein_Schwanzl_Vogt}, we show that the $\Loopt^r\Susp^r$-completion map participates in a derived equivalence between spaces and coalgebra spaces over the homotopical iterated suspension-loop comonad, after restricting to $1$-connected spaces. 

A guiding philosophy underlying our attack on Lawson's conjecture involves drawing an analogy between the spaces-level Hurewicz map and the spaces-level Freudenthal suspension map; a critical idea behind our strategy is to establish a higher Freudenthal suspension analog of Dundas' higher Hurewicz theorem \cite{Dundas_relative_K_theory}.

\subsection{Commuting suspensions with holim of a cobar construction}
Working with a tiny modification of the framework in Arone-Ching \cite{Arone_Ching_classification}, our main result amounts to proving that the left derived iterated suspension functor $\Susp^r$ commutes,
\begin{align}
\label{eq:key_technical_result}
  \Susp^r\holim\nolimits_\Delta \mathfrak{C}(Y)\wequiv
  \holim\nolimits_\Delta \Susp^r \mathfrak{C}(Y)
\end{align}
up to weak equivalence, with the right derived limit functor $\holim_\Delta $, when composed with the cosimplicial cobar construction $\mathfrak{C}$ associated to the homotopical comonad $\Kt$ and evaluated on $(1+r)$-connected $\Kt$-coalgebra spaces; this is the original form of Lawson's \cite{Lawson_conjecture} conjecture. In \cite{Blomquist_Harper_suspension_spectra} we address Lawson's conjecture in the infinite or limit case involving stabilization and suspension spectra.

\subsection{Organization of the paper}

In Section \ref{sec:outline_of_the_argument} we outline the argument of our main result. In Section \ref{sec:proofs} we review iterated suspension, the associated fundamental adjunction, and coalgebras over the homotopical comonad $\Kt$; we then prove higher Freudenthal suspension together with our main result. Sections \ref{sec:simplicial_structures}, \ref{sec:homotopy_theory_K_coalgebras}, and \ref{sec:derived_fundamental_adjunction} are background sections on simplicial structures, the homotopy theory of $\Kt$-coalgebras, and the derived fundamental adjunction, respectively, that are essential to understanding this paper. For the experts, it will suffice to read Sections \ref{sec:outline_of_the_argument} and \ref{sec:proofs} for a complete proof of the main result.

\subsection*{Acknowledgments}

The authors would like to thank Michael Ching for helpful suggestions and useful remarks throughout this project. The second author would like to thank Bjorn Dundas, Bill Dwyer, Haynes Miller, and Crichton Ogle for useful suggestions and remarks and Lee Cohn, Mike Hopkins, Tyler Lawson, and Nath Rao for helpful comments. The second author is grateful to Haynes Miller for a stimulating and enjoyable visit to the Massachusetts Institute of Technology in early spring 2015, and to Bjorn Dundas for a stimulating and enjoyable visit to the University of Bergen in late spring 2015, and for their invitations which made this possible. The first author was supported in part by National Science Foundation grant DMS-1510640.

\section{Outline of the argument}
\label{sec:outline_of_the_argument}

The purpose of this section is to outline the proof of our main result. Since the derived unit map is tautologically the $\Loopt^r\Susp^r$-completion map $X'\rarrow {X'}^\wedge_{\Loopt^r\Susp^r}$, which is proved to be a weak equivalence on $1$-connected spaces in Bousfield \cite{Bousfield_cosimplicial_space}, establishing the main result reduces to proving that the derived counit map is a weak equivalence. 

The following connectivity estimates are proved in Section \ref{sec:cubical_diagrams_homotopical_analysis}.

\begin{thm}
\label{thm:estimating_connectivity_of_maps_in_tower_C_of_Y}
If $Y$ is a $(1+r)$-connected $\Kt$-coalgebra space and $n\geq 1$, then the natural map
\begin{align}
\label{eq:tower_map_from_n_th_stage_to_next_lower_stage}
  \holim_{\Delta^{\leq n}}\mathfrak{C}(Y)&\longrightarrow
  \holim_{\Delta^{\leq n-1}}\mathfrak{C}(Y)
\end{align}
is an $(n+2)$-connected map between $1$-connected objects.
\end{thm}

\begin{thm}
\label{thm:connectivities_for_map_into_n_th_stage}
If $Y$ is a $(1+r)$-connected $\Kt$-coalgebra space and $n\geq 0$, then
\begin{align}
\label{eq:canonical_map_needed_to_discuss}
  \holim\nolimits_\Delta \mathfrak{C}(Y)&\longrightarrow
  \holim\nolimits_{\Delta^{\leq n}}\mathfrak{C}(Y)\\
\label{eq:the_other_canonical_map_needed}
  \Susp^r\holim\nolimits_\Delta \mathfrak{C}(Y)&\longrightarrow
  \Susp^r\holim\nolimits_{\Delta^{\leq n}}\mathfrak{C}(Y)
\end{align}
the natural map \eqref{eq:canonical_map_needed_to_discuss} (resp. \eqref{eq:the_other_canonical_map_needed}) is an $(n+3)$-connected map between $1$-connected objects (resp. $(n+3+r)$-connected map between $(1+r)$-connected objects).
\end{thm}

\begin{proof}
Consider the first part. By Theorem \ref{thm:estimating_connectivity_of_maps_in_tower_C_of_Y} each of the maps in the holim tower $\{\holim_{\Delta^{\leq n}}\mathfrak{C}(Y)\}_n$, above level $n$, is at least $(n+3)$-connected. It follows that the map \eqref{eq:canonical_map_needed_to_discuss} is $(n+3)$-connected. The second part follows from the first part.
\end{proof}

The following theorem underlies our main technical result. The homotopical analysis is worked out in Section \ref{sec:cubical_diagrams_homotopical_analysis} by leveraging the strong uniform cartesian-ness estimates for iterations of the Freudenthal suspension map, applied to $X=\Loopt^r Y$, in the higher Freudenthal suspension theorem, together with the ``uniformity of faces'' behavior forced by the cosimplicial identities. 

\begin{thm}
\label{thm:connectivities_for_map_that_commutes_Z_into_inside_of_holim}
If $Y$ is a $(1+r)$-connected $\Kt$-coalgebra space and $n\geq 1$, then the natural map
\begin{align}
\label{eq:commuting_Z_past_holim_delta}
  \Susp^r\holim\nolimits_{\Delta^{\leq n}} \mathfrak{C}(Y)\longrightarrow
  \holim\nolimits_{\Delta^{\leq n}} \Susp^r\,\mathfrak{C}(Y),
\end{align}
is $(n+5+r)$-connected; the map is a weak equivalence for $n=0$.
\end{thm}

The following is a corollary of these connectivity estimates.

\begin{thm}
\label{thm:ZC_commutes_with_desired_holim}
If $Y$ is a $(1+r)$-connected $\Kt$-coalgebra space, then the natural maps
\begin{align}
\label{eq:ZC_commutes_with_desired_holim}
  \Susp^r\holim\nolimits_\Delta \mathfrak{C}(Y)&\xrightarrow{\wequiv}
  \holim\nolimits_\Delta \Susp^r\,\mathfrak{C}(Y)\xrightarrow{\wequiv}Y
\end{align}
are weak equivalences.
\end{thm}

\begin{proof}
Consider the left-hand map; it suffices to verify that the connectivities of the natural maps \eqref{eq:the_other_canonical_map_needed} and \eqref{eq:commuting_Z_past_holim_delta}
 are strictly increasing with $n$, and Theorems \ref{thm:connectivities_for_map_into_n_th_stage} and \ref{thm:connectivities_for_map_that_commutes_Z_into_inside_of_holim} complete the proof. Consider the case of the right-hand map. Since $\Susp^r\, \mathfrak{C}(Y)\wequiv F\Susp^r\, \mathfrak{C}(Y)$ and the latter is isomorphic to the cosimplicial cobar construction $\Cobar(F\K,F\K,FY)$, which has extra codegeneracy maps $s^{-1}$ (Dwyer-Miller-Neisendorfer \cite[6.2]{Dwyer_Miller_Neisendorfer}), it follows from the cofinality argument in Dror-Dwyer \cite[3.16]{Dror_Dwyer_long_homology} that the right-hand map in \eqref{eq:ZC_commutes_with_desired_holim} is a weak equivalence.
\end{proof}

\begin{proof}[Proof of Theorem \ref{MainTheorem}]
The natural map
$
  \Susp^r\holim\nolimits_\Delta \mathfrak{C}(Y)\rarrow Y
$
is a weak equivalence since this is the composite \eqref{eq:ZC_commutes_with_desired_holim}; Theorem \ref{thm:ZC_commutes_with_desired_holim} completes the proof.
\end{proof}

\section{Homotopical analysis}
\label{sec:proofs}

The purpose of this section is to recall iterated suspension, the associated fundamental adjunction, and coalgebras over the associated homotopical comonad $\Kt$, and then to prove Theorems \ref{thm:estimating_connectivity_of_maps_in_tower_C_of_Y} and \ref{thm:connectivities_for_map_that_commutes_Z_into_inside_of_holim}. 

\subsection{Iterated suspension and the fundamental adjunction}

The fundamental adjunction naturally arises by observing that $\Susp^r$ is equipped with a coaction over the comonad $\K$ associated to the $(\Susp^r, \Loop^r)$ adjunction; this observation, which remains true for any adjunction provided that the indicated limits below exist, forms the basis of the homotopical descent ideas appearing in Hess \cite{Hess} and subsequently in Francis-Gaitsgory \cite{Francis_Gaitsgory}.

Consider any pointed spaces $X,Y$ and recall that the iterated suspension space $\Sigma^r(X):=S^r\Smash X$ and iterated loop space $\Omega^r(Y):=\hombold_*(S^r,Y)$ functors fit into the left-hand adjunction
\begin{align}
\label{eq:homology_adjunction}
\xymatrix{
  \Space_*\ar@<0.5ex>[r]^-{\Susp^r} &
  \Space_*\ar@<0.5ex>[l]^-{\Loop^r}
}
\quad\quad
\xymatrix{
  \Space_*\ar@<0.5ex>[r]^-{\Susp} &
  \Space_*\ar@<0.5ex>[l]^-{\Loop}\ \cdots\ 
  \Space_*\ar@<0.5ex>[r]^-{\Susp} &
  \Space_*\ar@<0.5ex>[l]^-{\Loop}
}
\quad\quad
\text{($r$ copies)}
\end{align}
with left adjoint on top; here, $S^r:=(S^1)^{\wedge r}$ for $r\geq 1$ where $S^1:=\Delta[1]/\partial\Delta[1]$. Note that the left-hand adjunction is naturally isomorphic to the right-hand $r$-fold iteration of the suspension-loop adjunction, by uniqueness of adjoints up to natural isomorphism. Associated to the adjunction in \eqref{eq:homology_adjunction} is the monad $\Loop^r\Susp^r$ on pointed spaces $\Space_*$ and the comonad $\K:=\Susp^r\Loop^r$ on pointed spaces $\Space_*$ of the form
\begin{align}
\label{eq:TQ_homology_spectrum_functor_natural_transformations}
  \id\xrightarrow{\eta} \Loop^r\Susp^r&\quad\text{(unit)},\quad\quad
  &\id\xleftarrow{\varepsilon}\K& \quad\text{(counit)}, \\
  \notag
  \Loop^r\Susp^r\Loop^r\Susp^r\rarrow \Loop^r\Susp^r&\quad\text{(multiplication)},\quad\quad
  &\K\K\xleftarrow{m}\K& \quad\text{(comultiplication)}.
\end{align}
and it follows formally that there is a factorization of adjunctions of the form
\begin{align}
\label{eq:factorization_of_adjunctions_Z_homology}
\xymatrix{
  \Space_*\ar@<0.5ex>[r]^-{\Susp^r} &
  \coAlgK\ar@<0.5ex>[r]\ar@<0.5ex>[l]^-{\lim_\Delta C} &
  \Space_*\ar@<0.5ex>[l]^-{\K} 
}
\end{align}
with left adjoints on top and $\coAlgK\rarrow\Space_*$ the forgetful functor. In particular, the iterated suspension space $\Susp^r X$ is naturally equipped with a $\K$-coalgebra structure. To understand the comparison in \eqref{eq:factorization_of_adjunctions_Z_homology} between $\Space_*$ and $\coAlgK$ it suffices to note that $\lim_\Delta C(Y)$ is naturally isomorphic to an equalizer of the form
\begin{align*}
  \lim_\Delta C(Y)\Iso
  \lim\Bigl(
  \xymatrix{
    \Loop^r Y\ar@<0.5ex>[r]^-{d^0}\ar@<-0.5ex>[r]_-{d^1} &
    \Loop^r\K Y
  }
  \Bigr)
\end{align*}
where $d^0=m\id$, $d^1=\id m$, $\function{m}{\Loop^r}{\Loop^r\K=\Loop^r\Susp^r\Loop^r}$ denotes the $\K$-coaction map on $\Loop^r$ (defined by $m:=\eta\id$), and $\function{m}{Y}{\K Y}$ denotes the $\K$-coaction map on $Y$.

\begin{defn}
\label{defn:cobar_construction}
Let $Y$ be a $\K$-coalgebra space. The \emph{cosimplicial cobar construction} $C(Y):=\Cobar(\Loop^r,\K,Y)$ in $(\Space_*)^{\Delta}$ looks like
\begin{align}
\label{eq:cobar_construction}
&\xymatrix{
  C(Y): \quad\quad
  \Loop^r Y\ar@<0.5ex>[r]^-{d^0}\ar@<-0.5ex>[r]_-{d^1} &
  \Loop^r \K Y
  \ar@<1.0ex>[r]\ar[r]\ar@<-1.0ex>[r] &
  \Loop^r\K\K Y
  \cdots
}
\end{align}
(showing only the coface maps) and is defined objectwise by $C(Y)^n:=\Loop^r\K^n Y$ with the obvious coface and codegeneracy maps; see, for instance, the face and degeneracy maps in the simplicial bar constructions described in Gugenheim-May \cite[A.1]{Gugenheim_May} or May \cite[Section 7]{May_classifying_spaces}, and dualize. For instance, in \eqref{eq:cobar_construction} the indicated coface maps are defined by $d^0:=m\id$ and $d^1:=\id m$.
\end{defn}

\subsection{Coalgebras over the homotopical comonad $\Kt$}

A useful first step will be to interpret the cosimplicial $\Loopt^r\Susp^r$-resolution of $X$ in terms of a cosimplicial cobar construction that naturally arises as a ``fattened'' version of \eqref{eq:cobar_construction}; this naturally leads to the notion of $\Kt$-coalgebra exploited in Cohn \cite{Cohn}.

\begin{defn}
\label{defn:fibrant_replacement}
Denote by $\function{\eta}{\id}{F}$ and $\function{m}{FF}{F}$ the unit and multiplication maps of the simplicial fibrant replacement monad $F=Sing(|-|)$ on pointed spaces $\Space_*$; see, for instance, Dundas-Goodwillie-McCarthy \cite{Dundas_Goodwillie_McCarthy} and Goerss-Jardine \cite{Goerss_Jardine} (compare with \cite[8.2]{Ching_Harper}); it is shown in Blumberg-Riehl \cite[6.1]{Blumberg_Riehl} that simplicial fibrant replacement monads are available in many homotopical contexts. It follows that $\Loopt^r:=\Loop^r F$ and $\Kt:=\K F$ are the derived functors of $\Loop^r$ and $\K$, respectively. The comultiplication $\function{m}{\Kt}{\Kt\Kt}$ and counit $\function{\varepsilon}{\Kt}{F}$ maps are defined by the composites
\begin{align}
\label{eq:comultiplication_K_tilde}
  &\K F\xrightarrow{m\id}\K\K F=
  \K\id\K F\xrightarrow{\id\eta\id\id}
  \K F\K F\\
  \label{eq:counit_K_tilde}
  &\K F\xrightarrow{\varepsilon\id}\id F=F
\end{align}
respectively.
\end{defn}

It is shown in Blumberg-Riehl \cite[4.2, 4.4]{Blumberg_Riehl}, and subsequently exploited in Cohn \cite{Cohn}, that the derived functor $\Kt:=\K F$ of the comonad $\K$ is very nearly a comonad itself with the structure maps $\function{m}{\Kt}{\Kt\Kt}$ and $\function{\varepsilon}{\Kt}{F}$ above. For instance, it is proved in \cite{Blumberg_Riehl} that $\Kt$ defines a comonad on the homotopy category of $\Space_*$, which is a reflection of the the fact that $\Kt$ has the structure of a highly homotopy coherent comonad (see \cite{Blumberg_Riehl}); in particular, $\Kt$ has a strictly coassociative comultiplication $\function{m}{\Kt}{\Kt\Kt}$ and satisfies left and right counit identities up to factors of $F\wequiv\id$.

\begin{rem}
Associated to the adjunction $(\Susp^r,\Loop^r)$ is a left $\K$-coaction (or $\K$-coalgebra structure) $\function{m}{\Susp^r X}{\K \Susp^r X}$ on $\Susp^r X$, defined by $m=\id\eta\id$), for any $X\in\Space_*$. This map induces a corresponding left $\Kt$-coaction $\function{m}{\Susp^r X}{\Kt\Susp^r X}$ that is the composite 
\begin{align*}
  \Susp^r X\xrightarrow{m}
  \K\Susp^r X=\K\id\Susp^r X
  \xrightarrow{}\K F\Susp^r X
\end{align*}
\end{rem}

The following notion of a homotopy $\Kt$-coalgebra, exploited in Cohn \cite{Cohn}, captures exactly the left $\Kt$-coaction structure that iterated suspension $\Susp^r X$ of a pointed space $X$ satisfies; this is precisely the structure being encoded by the cosimplicial $\Loopt^r\Sigma^r$ resolution \eqref{eq:iterated_loop_suspension_resolution_introduction}. 
 
\begin{defn}
A \emph{homotopy $\Kt$-coalgebra} (or $\Kt$-coalgebra, for short) is a $Y\in\Space_*$ together with a map $\function{m}{Y}{\Kt Y}$ in $\Space_*$ such that the following diagrams
\begin{align*}
  \xymatrix{
  Y\ar[r]^-{m}\ar[d]_-{m} & \Kt Y\ar[d]^-{m\id}\\
  \Kt Y\ar[r]_-{\id m} & \Kt\Kt Y
  }\quad\quad\quad
  \xymatrix{
    FY\ar[r]^-{\id m}\ar@{=}[d] & F\Kt Y\ar[d]^-{(*)}\\
     FY\ar@{=}[r] & FY
  }
\end{align*}
commute; here, the map $(*)$ is the composite
$
  F\Kt Y\xrightarrow{\id\varepsilon\id}FFY\xrightarrow{m\id}FY
$.
\end{defn}

\begin{defn}
\label{defn:cobar_construction_fattened}
Let $Y$ be a $\Kt$-coalgebra. The \emph{cosimplicial cobar construction} $\mathfrak{C}(Y):=\Cobar(\Loopt^r,\Kt,Y)$ in $(\Space_*)^{\Delta}$ looks like
\begin{align}
\label{eq:cobar_construction_fattened_up}
\xymatrix{
  \mathfrak{C}(Y): \quad\quad
  \Loopt^r Y\ar@<0.5ex>[r]^-{d^0}\ar@<-0.5ex>[r]_-{d^1} &
  \Loopt^r\Kt Y
  \ar@<1.0ex>[r]\ar[r]\ar@<-1.0ex>[r] &
  \Loopt^r\Kt\Kt Y
  \cdots
}
\end{align}
(showing only the coface maps) and is defined objectwise by $\mathfrak{C}(Y)^n:=\Loopt^r\Kt^n Y=\Loop^r F(\K F)^n Y$ with the obvious coface and codegeneracy maps; for instance, in \eqref{eq:cobar_construction_fattened_up} the indicated coface maps are defined by $d^0:=m\id$ and $d^1:=\id m$; compare with \eqref{eq:cobar_construction}.
\end{defn}

The cosimplicial resolution \eqref{eq:iterated_loop_suspension_resolution_introduction} of a pointed space $X$ with respect to $\Loopt^r\Susp^r$, built by iterating the spaces-level stabilization map \eqref{eq:freudenthal_suspension_map_spaces_level_introduction}, is naturally isomorphic to the map $X\rightarrow\mathfrak{C}(\Susp^r X)$; conceptually, the homotopical comonad $\Kt$ naturally encodes the spaces-level co-operations on the iterated suspension spaces.

\begin{rem}
The derived functor $\Loopt^r$ has a naturally occurring right $\Kt$-coaction map $\function{m}{\Loopt^r}{\Loopt^r\Kt}$, defined by the composite
\begin{align*}
  \Omega^rF\xrightarrow{m\id}\Omega^r\K F
  =\Omega^r\id\Kt\xrightarrow{\id\eta\id}\Omega^r F\Kt,
\end{align*}
that makes the following diagrams
\begin{align*}
  \xymatrix{
    \Loopt^r\ar[r]^-{m}\ar[d]_-{m} & 
    \Loopt^r\Kt\ar[d]^-{m\id}\\
    \Loopt^r\Kt\ar[r]_-{\id m} & \Loopt^r\Kt\Kt
  }\quad\quad\quad
  \xymatrix{
    \Loopt^r\ar[r]^-{m}\ar@{=}[d] &
    \Loopt^r\Kt\ar[d]^-{(**)}\\
    \Loopt^r\ar@{=}[r] & \Loopt^r
  }
\end{align*}
commute; here, the map $(**)$ is the composite
$
  \Loop^r F\Kt\xrightarrow{\id\id\varepsilon}\Loop^r FF\xrightarrow{\id m}\Loop^r F
$.
\end{rem}

\begin{rem}
The homotopical comonad $\Kt$ makes the following diagrams
\begin{align*}
  \xymatrix{
    \Kt\ar[r]^-{m}\ar[d]_-{m} & 
    \Kt\Kt\ar[d]^-{m\id}\\
    \Kt\Kt\ar[r]_-{\id m} & \Kt\Kt\Kt
  }\quad\quad\quad
  \xymatrix{
    F\Kt\ar[r]^-{\id m}\ar@{=}[d] & F\Kt \Kt\ar[d]^-{(*)}\\
     F\Kt\ar@{=}[r] & F\Kt
  }\quad\quad\quad
  \xymatrix{
    \Kt\ar[r]^-{m}\ar@{=}[d] &
    \Kt\Kt\ar[d]^-{(**)}\\
    \Kt\ar@{=}[r] & \Kt
  }
\end{align*}
commute; here, the map $(*)$ is the composite
$
  F\Kt \Kt\xrightarrow{\id\varepsilon\id}FF\Kt\xrightarrow{m\id}F\Kt
$ and the map $(**)$ is the composite
$
  \K F\Kt\xrightarrow{\id\id\varepsilon}\K FF\xrightarrow{\id m}\K F
$.
\end{rem}

\begin{rem}
The counit map \eqref{eq:counit_K_tilde} is identical to the composite
\begin{align*}
  \K F=\id\K F\xrightarrow{\eta\id\id} F\K F\xrightarrow{\id\varepsilon\id} F\id F=FF\xrightarrow{m} F
\end{align*}
This may be useful when comparing with \cite{Blumberg_Riehl}.
\end{rem}

\subsection{Higher Freudenthal suspension}

In this section we prove Theorem \ref{thm:higher_freudenthal_suspension};  we will freely make use of the notations and definitions associated to cubical diagrams in Goodwillie \cite{Goodwillie_calculus_2} and Ching-Harper \cite{Ching_Harper}.

Recall the following from Dundas \cite{Dundas_relative_K_theory} and Dundas-Goodwillie-McCarthy \cite{Dundas_Goodwillie_McCarthy}.

\begin{defn}
Let $\function{f}{\NN}{\NN}$ be a function and $W$ a finite set. A $W$-cube $\capX$ is \emph{$f$-cartesian} (resp. \emph{$f$-cocartesian}) if each $d$-subcube of $\capX$ is $f(d)$-cartesian (resp. $f(d)$-cocartesian); here, $\NN$ denotes the non-negative integers.
\end{defn}

The following proposition will be helpful in organizing the proof of the higher Freudenthal suspension theorem below; compare with \cite[A.8.3]{Dundas_Goodwillie_McCarthy}.

\begin{prop}[Uniformity correspondence]
\label{prop:uniformity_correspondence}
Let $k\geq 1$ and $W$ a finite set. A $W$-cube of pointed spaces is $(k(\id+1)+1)$-cartesian if and only if it is $((k+1)(\id+1)-1)$-cocartesian.
\end{prop}

\begin{proof}
This is tautologically true for $|W|=0,1$. Let $n\geq 2$. Assume the statement is true for all $|W|<n$; let's verify it is true for $|W|=n$. Let $W=\{1,\dotsb,n\}$ and suppose $\capX$ is a $W$-cube of pointed spaces. Assume that $\capX$ is $(k(\id+1)+1)$-cartesian; let's verify $\capX$ is $((k+1)(\id+1)-1)$-cocartesian. By the induction hypothesis, it suffices to verify that $\capX$ is $(k(n+1)+n)$-cocartesian; this follows easily from Goodwillie's \cite[2.6]{Goodwillie_calculus_2} higher dual Blakers-Massey theorem. Conversely, assume that $\capX$ is $((k+1)(\id+1)-1)$-cocartesian; let's verify $\capX$ is $(k(\id+1)+1)$-cartesian. By the induction hypothesis, it suffices to verify that $\capX$ is $(k(n+1)+1)$-cartesian; this follows easily from Goodwillie's \cite[2.5]{Goodwillie_calculus_2} higher Blakers-Massey theorem.
\end{proof}

The following theorem plays a key role in our homotopical analysis of the derived counit map below; it also provides an alternate proof, with stronger estimates, of the result in Bousfield \cite{Bousfield_cosimplicial_space} that the $\Loopt^r\Susp^r$-completion map $X\rarrow X^\wedge_{\Loopt^r\Susp^r}$ is a weak equivalence for any $1$-connected space $X$. Our argument is closely related to Dundas \cite[2.6]{Dundas_relative_K_theory}.

\begin{thm}[Higher Freudenthal suspension theorem]
\label{thm:higher_freudenthal_suspension}
Let $k\geq 1$, $W$ a finite set, and $\capX$ a $W$-cube of pointed spaces. If $\capX$ is $(k(\id+1)+1)$-cartesian, then so is $\capX\rarrow\tilde{\Loop}^r\Susp^r\capX$.
\end{thm}

\begin{proof}
Consider the case $|W|=0$. Suppose $\capX$ is a $W$-cube and $\capX_\emptyset$ is $k$-connected. We know by Freudenthal suspension, which can be understood as a consequence of the Blakers-Massey theorem (see, for instance, \cite[A.8.2]{Dundas_Goodwillie_McCarthy}), that the map $\capX_\emptyset\rarrow\tilde{\Loop}\Susp\capX_\emptyset$ is $(2k+1)$-connected. More generally, it follows by repeated application of Freudenthal suspension that the composite $\capX_\emptyset\rarrow\tilde{\Loop}^r\Susp^r\capX_\emptyset$ is a $(2k+1)$-connected map between $k$-connected spaces. 

Consider the case $|W|\geq 1$. Suppose $\capX$ is a $W$-cube and $\capX$ is $(k(\id+1)+1)$-cartesian. Let's verify that $\capX\rarrow\tilde{\Loop}^r\Susp^r\capX$ is $(k(\id+1)+1)$-cartesian $(|W|+1)$-cube. It suffices to assume that $\capX$ is a cofibration $W$-cube; see \cite[1.13]{Goodwillie_calculus_2} and \cite[3.4]{Ching_Harper}. Let $C$ be the iterated cofiber of $\capX$ and $\capC$ the $W$-cube defined objectwise by $\capC_V={*}$ for $V\neq W$ and $\capC_W=C$. Then $\capX\rarrow\capC$ is $\infty$-cocartesian. Consider the commutative diagram
\begin{align}
\label{eq:diagram_of_cubes_for_higher_freudenthal}
\xymatrix{
  \capX\ar[r]\ar[d]_(0.4){(*)} & 
  \capC\ar[d]\\
  \tilde{\Loop}^r\Susp^r\capX\ar[r] & 
  \tilde{\Loop}^r\Susp^r\capC
}
\end{align}
of $|W|$-cubes. 

Let's verify that $(*)$ is $(k(|W|+2)+1)$-cartesian as a $(|W|+1)$-cube of pointed spaces. We know that $\capX$ is $((k+1)(\id+1)-1)$-cocartesian by Proposition \ref{prop:uniformity_correspondence}, and in particular, $C$ is $((k+1)(|W|+1)-1)$-connected. For $d<|W|$, any $(d+1)$ dimensional subcube of $\capX$ is $((k+1)(d+2)-1)=((k+1)(d+1)+k)$-cocartesian and any $d$ dimensional subcube of $\capX$ is $((k+1)(d+1)-1)$-cocartesian.  So if $\capX|T$ is some $d$-subcube of $\capX$ with $T$ not containing the terminal set $W$, then $\capX|T\rarrow\capC|T={*}$ is $(k+1)(d+1)$-cocartesian by \cite[1.7]{Goodwillie_calculus_2}. Furthermore, even if $T$ contains the terminal set $W$, we know that $\capX|T\rarrow\capC|T$ is still $(k+1)(d+1)$-cocartesian by \cite[1.7]{Goodwillie_calculus_2}; this is because $(k+1)(d+1)<(k+1)(|W|+1)-1$ since $k\geq 1$ and $d<|W|$. Hence $\capX|T\rarrow\capC|T$ is $(k+1)(d+1)$-cocartesian for any $d$-subcube $\capX|T$ of $\capX$. It follows easily from higher Blakers-Massey \cite[2.5]{Goodwillie_calculus_2} that $\capX\rarrow\capC$ is $(k(|W|+2)+1)$-cartesian. Similarly, it follows that $\Susp^r\capX\rarrow\Susp^r\capC$ is $(k(|W|+2)+1+2r)$-cartesian and hence $\tilde{\Loop}^r\Susp^r\capX\rarrow\tilde{\Loop}^r\Susp^r\capC$ is $(k(|W|+2)+1+r)$-cartesian. Also, $\capC\rarrow\tilde{\Loop}^r\Susp^r\capC$ is at least $(k(|W|+2)+1)$-cartesian since $C\rarrow\tilde{\Loop}^r\Susp^r C$ is $(2[(k+1)(|W|+1)-1]+1)$-connected by Freudenthal suspension; this is because the cartesian-ness of $\capC\rarrow\tilde{\Loop}^r\Susp^r\capC$ is the same as the connectivity of the map $\tilde{\Loop}^{|W|}\capC\rarrow\tilde{\Loop}^{|W|}\tilde{\Loop}^r\Susp^r\capC$ (by considering iterated homotopy fibers).

Putting it all together, it follows from diagram \eqref{eq:diagram_of_cubes_for_higher_freudenthal} and \cite[1.8]{Goodwillie_calculus_2} that the map $(*)$ is $(k(|W|+2)+1)$-cartesian; this is because $k(|W|+2)+1<k(|W|+2)+1+r$. Doing this also on all subcubes gives the result.
\end{proof}

\subsection{Homotopical estimates and codegeneracy cubes}

\label{sec:cubical_diagrams_homotopical_analysis}

In this section we prove Theorems \ref{thm:estimating_connectivity_of_maps_in_tower_C_of_Y} and \ref{thm:connectivities_for_map_that_commutes_Z_into_inside_of_holim}.

The following is a dressed up form of Bousfield-Kan's calculation \cite[X.6.3]{Bousfield_Kan} of the layers of the $\Tot$ tower; it also highlights the homotopical significance of the codegeneracy $n$-cubes $\capY_n$.

\begin{prop}
\label{prop:iterated_homotopy_fibers_calculation}
Let $Z$ be a cosimplicial pointed space and $n\geq 0$. There are natural zigzags of weak equivalences
\begin{align*}
  \hofib(\holim_{\Delta^{\leq n}}Z\rarrow\holim_{\Delta^{\leq n-1}}Z)
  \wequiv\Omega^n(\iter\hofib)\capY_n
\end{align*}
where $\capY_n$ denotes the canonical $n$-cube built from the codegeneracy maps of
\begin{align*}
\xymatrix{
  Z^0 &
  Z^1
  \ar[l]_-{s^0} &
  Z^2\ar@<-0.5ex>[l]_-{s^0}\ar@<0.5ex>[l]^-{s^1}
  \ \cdots\ Z^n
}
\end{align*}
 the $n$-truncation of $Z$; in particular, $\capY_0$ is the object (or $0$-cube) $Z^0$. We often refer to $\capY_n$ as the \emph{codegeneracy} $n$-cube associated to $Z$.
\end{prop}

The following is proved in Carlsson \cite[Section 6]{Carlsson}, Dugger \cite{Dugger_homotopy_colimits}, and Sinha \cite[6.7]{Sinha_cosimplicial_models}, and plays a key role in this paper; see also  Dundas-Goodwillie-McCarthy \cite{Dundas_Goodwillie_McCarthy} and Munson-Volic \cite{Munson_Volic_book_project}; it was exploited early on by Hopkins \cite{Hopkins_iterated_suspension}.

\begin{prop}
\label{prop:left_cofinality_truncated_delta}
Let $n\geq 0$. The composite
\begin{align*}
  \capP_0([n])\Iso P\Delta[n]\longrightarrow\Delta_\res^{\leq n}
  \subset\Delta^{\leq n}
\end{align*}
is left cofinal (i.e., homotopy initial). Here, $\capP_0([n])$ denotes the poset of all nonempty subsets of $[n]$ and $P\Delta[n]$ denotes the poset of non-degenerate simplices of the standard $n$-simplex $\Delta[n]$; see \cite[III.4]{Goerss_Jardine}.
\end{prop}

\begin{prop}
\label{prop:punctured_cube_calculation_of_holim_truncated_delta}
If $X\in\M^\Delta$ is objectwise fibrant, then the natural maps
\begin{align*}
  \holim\nolimits_{\Delta^{\leq n}}^\BK X&\xrightarrow{\wequiv}
  \holim\nolimits_{P\Delta[n]}^\BK X\Iso
  \holim\nolimits_{\capP_0([n])}^\BK X
\end{align*}
in $\M$ are weak equivalences.
\end{prop}

\begin{rem}
We follow the conventions and definitions in Bousfield-Kan \cite{Bousfield_Kan} and Ching-Harper \cite{Ching_Harper_derived_Koszul_duality} for the various models of homotopy limits; see also \cite{Blomquist_Harper_integral_chains}.
\end{rem}

\begin{defn}
\label{defn:the_wide_tilde_construction}
Let $Z$ be a cosimplicial pointed space and $n\geq 0$. Assume that $Z$ is objectwise fibrant and denote by $\function{Z}{\capP_0([n])}{\Space_*}$ the composite
\begin{align*}
  \capP_0([n])\rightarrow\Delta^{\leq n}
  \rightarrow\Delta
  \rightarrow\Space_*
\end{align*}
The \emph{associated $\infty$-cartesian $(n+1)$-cube built from $Z$}, denoted $\function{\widetilde{Z}}{\capP([n])}{\Space_*}$, is defined objectwise by
\begin{align*}
  \widetilde{Z}_V :=
  \left\{
    \begin{array}{rl}
    \holim^\BK_{T\neq\emptyset}Z_T,&\text{for $V=\emptyset$,}\\
    Z_V,&\text{for $V\neq\emptyset$}.
    \end{array}
  \right.
\end{align*}
It is important to note (Proposition \ref{prop:punctured_cube_calculation_of_holim_truncated_delta}) that there are natural weak equivalences
\begin{align*}
  \holim\nolimits_{\Delta^{\leq n}}Z\wequiv
  \holim\nolimits^\BK_{T\neq\emptyset}Z_T=\widetilde{Z}_\emptyset
\end{align*}
in $\Space_*$.
\end{defn}

\begin{prop}[Uniformity of faces]
\label{prop:comparing_faces_of_coface_cube_with_codegeneracy_cube}
Let $Z\in(\Space_*)^\Delta$ and $n\geq 0$. Assume that $Z$ is objectwise fibrant. Let $\emptyset\neq T\subset[n]$ and $t\in T$. Then there is a weak equivalence
\begin{align*}
  (\iter\hofib)\partial_{\{t\}}^T\widetilde{Z}\wequiv
  \Omega^{|T|-1}(\iter\hofib)\capY_{|T|-1}
\end{align*}
in $\Space_*$, where $\capY_{|T|-1}$ denotes the codegeneracy $(|T|-1)$-cube associated to $Z$.
\end{prop}

\begin{proof}
This is proved in \cite{Ching_Harper_derived_Koszul_duality}; compare \cite[3.4]{Goodwillie_calculus_3} and \cite[7.2]{Sinha_cosimplicial_models}.
\end{proof}

\begin{thm}
\label{thm:cocartesian_and_cartesian_estimates}
Let $Y$ be a $\Kt$-coalgebra space and $n\geq 1$. Consider the $\infty$-cartesian $(n+1)$-cube $\widetilde{\mathfrak{C}(Y)}$ in $\Space_*$ built from $\mathfrak{C}(Y)$. If $Y$ is $(1+r)$-connected, then
\begin{itemize}
\item[(a)] the cube $\widetilde{\mathfrak{C}(Y)}$ is $(2\cdot\id+1)$-cocartesian  and $(2n+5)$-cocartesian in $\Space_*$,
\item[(b)] the cube $\Susp^r\widetilde{\mathfrak{C}(Y)}$ is $(2\cdot\id+1+r)$-cocartesian and $(2n+5+r)$-cocartesian in $\Space_*$,
\item[(c)] the cube $\Susp^r\widetilde{\mathfrak{C}(Y)}$ is $(n+5+r)$-cartesian in $\Space_*$.
\end{itemize}
\end{thm}

\begin{proof}
Consider part (a) and let $W=[n]$. We use higher dual Blakers-Massey in Goodwillie \cite[2.6]{Goodwillie_calculus_2} to estimate how close the $W$-cube $\widetilde{\mathfrak{C}(Y)}$ and its subcubes in $\Space_*$ are to being cocartesian. We know from higher Freudenthal suspension (Theorem \ref{thm:higher_freudenthal_suspension}) on iterations of the Freudenthal suspension map applied to $\tilde{\Loop}^r Y$, together with the uniformity enforced by Proposition \ref{prop:comparing_faces_of_coface_cube_with_codegeneracy_cube}, that for each nonempty subset $V\subset W$, the $V$-cube $\partial_{W-V}^W\widetilde{\mathfrak{C}(Y)}$ is $(|V|+2)$-cartesian; since it is $\infty$-cartesian by construction when $V=W$, it follows immediately from higher dual Blakers-Massey \cite[2.6]{Goodwillie_calculus_2} that $\widetilde{\mathfrak{C}(Y)}$ is $(2n+5)$-cocartesian. Similarly, it follows that $\widetilde{\mathfrak{C}(Y)}$ is $(\id+2)$-cartesian, and hence by the uniformity correspondence (Proposition \ref{prop:uniformity_correspondence}) we know that $\widetilde{\mathfrak{C}(Y)}$ is $(2\cdot\id+1)$-cocartesian in $\Space_*$ which finishes the proof of part (a). Part (b) is an easy consequence of part (a) since $\function{\Susp^r}{\Space_*}{\Space_*}$ increases cocartesian-ness by $r$. Part (c) follows immediately from higher Blakers-Massey in Goodwillie \cite[2.5]{Goodwillie_calculus_2} together with the cocartesian-ness estimates in part (b).
\end{proof}

\begin{proof}[Proof of Theorem \ref{thm:connectivities_for_map_that_commutes_Z_into_inside_of_holim}]
We want to estimate how connected the comparison map
\begin{align*}
  \Susp^r\holim\nolimits_{\Delta^{\leq n}} \mathfrak{C}(Y)\longrightarrow
  \holim\nolimits_{\Delta^{\leq n}} \Susp^r\,\mathfrak{C}(Y),
\end{align*}
is, which is equivalent to estimating how cartesian $\Susp^r\widetilde{\mathfrak{C}(Y)}$ is; Theorem \ref{thm:cocartesian_and_cartesian_estimates}(c) completes the proof.
\end{proof}

The following proposition gives the connectivity estimates that we need.

\begin{prop}
\label{prop:iterated_hofiber_codegeneracy_cube}
Let $Y$ be a $\Kt$-coalgebra space and $n\geq 1$. Denote by $\capY_n$ the codegeneracy $n$-cube associated to the cosimplicial cobar construction $\mathfrak{C}(Y)$ of $Y$. If $Y$ is $(1+r)$-connected, then the total homotopy fiber of $\capY_n$ is $(2n+1)$-connected.
\end{prop}

\begin{proof}
This follows immediately from the proof of Theorem \ref{thm:cocartesian_and_cartesian_estimates}, together with Proposition \ref{prop:comparing_faces_of_coface_cube_with_codegeneracy_cube}.
\end{proof}

\begin{proof}[Proof of Theorem \ref{thm:estimating_connectivity_of_maps_in_tower_C_of_Y}]
The homotopy fiber of the map \eqref{eq:tower_map_from_n_th_stage_to_next_lower_stage} is weakly equivalent to $\Loopt^n$ of the total homotopy fiber of the codegeneracy $n$-cube $\capY_{n}$ associated to $\mathfrak{C}(Y)$ by Proposition \ref{prop:iterated_homotopy_fibers_calculation}, hence by Proposition \ref{prop:iterated_hofiber_codegeneracy_cube} the map \eqref{eq:tower_map_from_n_th_stage_to_next_lower_stage} is $(n+2)$-connected.
\end{proof}

\subsection{Strong convergence}

The following is a corollary of the connectivity estimates in Theorem \ref{thm:estimating_connectivity_of_maps_in_tower_C_of_Y}; compare with the homotopy spectral sequence discussion in \cite{Bousfield_Kan, Goerss_Jardine}.

\begin{thm}
If $Y$ is a $(1+r)$-connected $\Kt$-coalgebra space, then the homotopy spectral sequence
\begin{align*}
  E^2_{-s,t} &= \pi^s\pi_t \mathfrak{C}(Y)
  \Longrightarrow
  \pi_{t-s}\holim\nolimits_\Delta \mathfrak{C}(Y)
\end{align*}
converges strongly; compare with \cite{Ching_Harper_derived_Koszul_duality}.
\end{thm}

\begin{proof}
This is because of the connectivity estimates in Theorem \ref{thm:estimating_connectivity_of_maps_in_tower_C_of_Y}.
\end{proof}

Compare with Bousfield-Kan \cite{Bousfield_Kan_spectral_sequence} and Bendersky-Curtis-Miller \cite{Bendersky_Curtis_Miller} and Bendersky-Thompson \cite{Bendersky_Thompson}.

\section{Background on simplicial structures}  
\label{sec:simplicial_structures}

The expert may wish to skim through, or skip entirely, this background section; here we recall the well known simplicial structures on pointed spaces that will be exploited in this paper.

\begin{defn}
\label{defn:simplicial_structure_pointed_spaces}
Let $X,X'$ be pointed spaces and $K$ a simplicial set. The \emph{tensor product} $X\tensordot K$ in $\Space_*$, \emph{mapping object} $\hombold_{\Space_*}(K,X)$ in $\Space_*$, and \emph{mapping space} $\Hombold_{\Space_*}(X,X')$ in $\sSet$ are defined by
\begin{align*}
  X\tensordot K &:=X\Smash K_+\\
  \hombold_{\Space_*}(K,X')&:=\hombold_*(K_+,X')\\
  \Hombold_{\Space_*}(X,X')_n &:= \hom_{\Space_*}(X\tensordot\Delta[n],X')
\end{align*}
where the \emph{pointed mapping space} $\hombold_*(X,X')$ in $\Space_*$ is $\Hombold_{\Space_*}(X,X')$ pointed by the constant map.  We will sometimes simply write $\Hombold$ and $\hombold$.
\end{defn}

The following is proved, for instance, in Goerss-Jardine \cite[II.3]{Goerss_Jardine}.

\begin{prop}
With the above definitions of mapping object, tensor product, and mapping space the category of pointed spaces $\Space_*$ is a simplicial model category.
\end{prop}

\begin{rem}
\label{rem:useful_adjunction_isomorphisms_simplicial_structure}
In particular, there are isomorphisms
\begin{align}
\label{eq:tensordot_adjunction_isomorphisms}
  \hom_{\Space_*}(X\tensordot K,X')
  &\Iso\hom_{\Space_*}(X,\hombold(K,X'))\\
\notag
  &\Iso\hom_\sSet(K,\Hombold(X,X'))
\end{align}
in $\Set$, natural in $X,K,X'$, that extend to isomorphisms
\begin{align*}
  \Hombold_{\Space_*}(X\tensordot K,X')
  &\Iso\Hombold_{\Space_*}(X,\hombold(K,X'))\\
  &\Iso\Hombold_\sSet(K,\Hombold(X,X'))
\end{align*}
in $\sSet$, natural in $X,K,X'$.
\end{rem}

\subsection{Simplicial natural transformations}

Recall that the iterated suspension-loop adjunction $(\Susp^r, \Loop^r)$ in \eqref{eq:homology_adjunction} is a Quillen adjunction with left adjoint on top; in particular, for $X,Y\in\Space_*$ there is an isomorphism
\begin{align}
\label{eq:hom_set_adjunction_change_of_operads_general}
  \hom_{\Space_*}(\Susp^r X,Y)\Iso\hom_{\Space_*}(X,\Loop^r Y)
\end{align}
in $\Set$, natural in $X,Y$.

The following proposition can be verified from Goerss-Jardine \cite[II.2.9]{Goerss_Jardine}.

\begin{prop}
\label{prop:useful_properties_of_the_adjunction}
Let $X,Y$ be pointed spaces and $K,L$ simplicial sets.  Then
\begin{itemize}
\item[(a)] there is a natural isomorphism
$
  \sigma\colon\thinspace \Susp^r(X)\tensordot K \xrightarrow{\Iso}\Susp^r(X\tensordot K)
$;
\item[(b)] there is an isomorphism
\begin{align*}
  \Hombold(\Susp^r X,Y)\Iso\Hombold(X,\Loop^r Y)
\end{align*}
in $\sSet$, natural in $X,Y$, that extends the adjunction isomorphism in  \eqref{eq:hom_set_adjunction_change_of_operads_general};
\item[(c)] there is an isomorphism
\begin{align*}
  \Loop^r\hombold(K,Y)\Iso\hombold(K,\Loop^r Y)
\end{align*}
in $\Space_*$, natural in $K,Y$.
\item[(d)] there is a natural map
$
  \function{\sigma}{\Loop^r(Y)\tensordot K}{\Loop^r(Y\tensordot K)}
$
induced by $\Loop^r$.
\item[(e)] the functors $\Susp^r$ and $\Loop^r$ are simplicial functors (Remark \ref{rem:simplicial_functors}) with the structure maps $\sigma$ of (a) and (d), respectively.
\end{itemize}
\end{prop}

\begin{rem}
\label{rem:simplicial_functors}
See Hirschhorn \cite[9.8.5]{Hirschhorn} for a useful reference on simplicial functors in homotopy theory.
\end{rem}

\begin{prop}
\label{prop:unit_and_counit_are_simplicial}
Consider the monad $\Loop^r\Susp^r$ on pointed spaces $\Space_*$ and the comonad $\Susp^r\Loop^r$ on pointed spaces $\Space_*$ associated to the adjunction $(\Susp^r,\Loop^r)$ in \eqref{eq:homology_adjunction}. The associated natural transformations
\begin{align*}
  \id\xrightarrow{\eta} \Loop^r\Susp^r &\quad\text{(unit)},\quad\quad
  &\id\xleftarrow{\varepsilon}\Susp^r\Loop^r & \quad\text{(counit)}, \\
  \Loop^r\Susp^r\Loop^r\Susp^r\rarrow \Loop^r\Susp^r &\quad\text{(multiplication)},\quad\quad
  &\Susp^r\Loop^r\Susp^r\Loop^r\xleftarrow{m}\Susp^r\Loop^r& \quad\text{(comultiplication)}
\end{align*}
are simplicial natural transformations.
\end{prop}

\begin{proof}
This is proved similar to \cite[Proof of 3.16]{Ching_Harper_derived_Koszul_duality}.
\end{proof}

\section{Background on the homotopy theory of $\Kt$-coalgebras}
\label{sec:homotopy_theory_K_coalgebras}

The purpose of this section is to recall the homotopy theory of $\Kt$-coalgebras developed in Arone-Ching \cite{Arone_Ching_classification}; we follow the slightly modified formulations exploited in Ching-Harper \cite{Ching_Harper} and Cohn \cite{Cohn}. The expert may wish to skim through, or skip entirely, this background section.

A morphism of $\Kt$-coalgebra spaces from $Y$ to $Y'$ is a map $\function{f}{Y}{Y'}$ in $\Space_*$ that makes the diagram
\begin{align}
\label{eq:commutative_square_defining_K_coalgebra_map}
\xymatrix{
  Y\ar[d]_-{f}\ar[r]^-{m} & \Kt Y\ar[d]^-{\id f}\\
  Y'\ar[r]_-{m} & \Kt Y'
}
\end{align}
in $\Space_*$ commute. This motivates the following homotopical cosimplicial resolution of $\Kt$-coalgebra maps from $Y$ to $Y'$.

\begin{defn}
\label{defn:derived_K_coalgebra_maps}
Let $Y,Y'$ be $\Kt$-coalgebra spaces. The object $\Hombold\bigl(Y,F\Kt^\bullet Y'\bigr)$ in $(\sSet)^\Delta$ looks like (showing only the coface maps)
\begin{align*}
\xymatrix{
  \Hombold(Y,FY')\ar@<0.5ex>[r]^-{d^0}\ar@<-0.5ex>[r]_-{d^1} &
  \Hombold\bigl(Y,F\Kt Y'\bigr)
  \ar@<1.0ex>[r]\ar[r]\ar@<-1.0ex>[r] &
  \Hombold\bigl(Y,F\Kt\Kt Y'\bigr)\cdots
}
\end{align*}
and is defined objectwise by 
\begin{align*}
  \Hombold\bigl(Y,F\Kt^\bullet Y'\bigr)^n:=
  \Hombold\bigl(Y,F\Kt^n Y'\bigr)=
  \Hombold\bigl(Y,F(\K F)^n Y'\bigr)
\end{align*}
 with the obvious coface and codegeneracy maps; see, Arone-Ching \cite[1.3]{Arone_Ching_classification}.
\end{defn}

\begin{defn}
\label{defn:realization_sSet}
The \emph{realization} functor $\function{|-|}{\sSet}{\CGHaus}$ for simplicial sets is defined objectwise by the coend $X\mapsto X \times_{\Delta}\Delta^{(-)}$; here, $\Delta^n$ in $\CGHaus$ denotes the topological standard $n$-simplex for each $n\geq 0$ (see \cite[I.1.1]{Goerss_Jardine}).
\end{defn}

\begin{defn}
Let $X,Y$ be pointed spaces. The mapping space functor $\Map_{\Space_*}$ is defined objectwise by realization
$
  \Map_{\Space_*}(X,Y)
  :=|\Hombold_{\Space_*}(X,Y)|
$
of the indicated simplicial set.
\end{defn}

\begin{defn}
Let $Y,Y'$ be $\Kt$-coalgebra spaces. The \emph{mapping spaces} of derived $\Kt$-coalgebra maps $\Hombold_{\coAlgKt}(Y,Y')$ in $\sSet$ and $\Map_{\coAlgKt}(Y,Y')$ in $\CGHaus$ are defined by the restricted totalizations
\begin{align*}
  \Hombold_{\coAlgKt}(Y,Y')
  &:=\Tot^\res\Hombold\bigl(Y,F\Kt^\bullet Y'\bigr)\\
  \Map_{\coAlgKt}(Y,Y')
  &:=\Tot^\res\Map\bigl(Y,F\Kt^\bullet Y'\bigr)
\end{align*}
of the indicated cosimplicial objects.
\end{defn}

\begin{rem}
\label{rem:intuition_behind_resolution}
Note that there are natural zigzags of weak equivalences
\begin{align*}
  \Hombold_{\coAlgKt}(Y,Y')
  \wequiv\holim_{\Delta}\Hombold\bigl(Y,F\Kt^\bullet Y'\bigr)
\end{align*}
\end{rem}

\begin{defn}
\label{defn:derived_K_coalgebra_map}
Let $Y,Y'$ be $\Kt$-coalgebra spaces. A \emph{derived $\Kt$-coalgebra map} $f$ of the form $Y\rarrow Y'$ is any map in $(\sSet)^{\Delta_\res}$ of the form
\begin{align*}
  \functionlong{f}{\Delta[-]}{\Hombold\bigl(Y,F\Kt^\bullet Y'\bigr)}.
\end{align*}
A \emph{topological derived $\Kt$-coalgebra map} $g$ of the form $Y\rarrow Y'$ is any map in $(\CGHaus)^{\Delta_\res}$ of the form
\begin{align*}
  \functionlong{g}{\Delta^\bullet}{\Map\bigl(Y,F\Kt^\bullet Y'\bigr)}.
\end{align*}
The \emph{underlying map} of a derived $\Kt$-coalgebra map $f$ is the map $\function{f_0}{Y}{FY'}$ that corresponds to the map $\function{f_0}{\Delta[0]}{\Hombold(Y,FY')}$. Every derived $\Kt$-coalgebra map $f$ determines a topological derived $\Kt$-coalgebra map $|f|$ by realization.
\end{defn}

\begin{defn}
The \emph{homotopy category} of $\Kt$-coalgebra spaces (compare, \cite[1.15]{Arone_Ching_classification}), denoted $\Ho(\coAlgKt)$, is the category with objects the $\Kt$-coalgebras and morphism sets $[X,Y]_\Kt$ from $X$ to $Y$ the path components
\begin{align*}
  [X,Y]_\Kt := \pi_0\Map_\coAlgKt(X,Y)
\end{align*}
of the indicated mapping spaces.
\end{defn}

\begin{defn}
\label{defn:weak_equivalence_of_K_coalgebras}
A derived $\Kt$-coalgebra map $f$ of the form $Y\rarrow Y'$ is a \emph{weak equivalence} if the underlying map $\function{f_0}{Y}{FY'}$ is a weak equivalence.
\end{defn}

\begin{prop}
\label{prop:weak_equivalence_if_and_only_if_iso_in_homotopy_category}
Let $Y,Y'$ be $\Kt$-coalgebra spaces. A derived $\Kt$-coalgebra map $f$ of the form $Y\rarrow Y'$ is a weak equivalence if and only if it represents an isomorphism in the homotopy category of $\Kt$-coalgebras.
\end{prop}

\section{Background on the derived fundamental adjunction}
\label{sec:derived_fundamental_adjunction}

The purpose of this section is to review the derived fundamental adjunction associated to the iterated suspension-loop adjunction; the expert may wish to skim through, or skip entirely, this background section.

The derived unit is the map of pointed spaces of the form $X\rarrow\holim_\Delta \mathfrak{C}(\Susp^r X)$ corresponding to the identity map $\function{\id}{\Susp^r X}{\Susp^r X}$; it is tautologically the $\Loopt^r\Susp^r$-completion map $X\rarrow X^\wedge_{\Loopt^r\Susp^r}$ studied in Bousfield \cite{Bousfield_cosimplicial_space}. 

\begin{defn}
\label{defn:derived_counit_map}
The \emph{derived counit map} associated to the fundamental adjunction \eqref{eq:fundamental_adjunction_comparing_spaces_with_coAlgK} is the derived $\Kt$-coalgebra map of the form $\Susp^r\holim_\Delta \mathfrak{C}(Y)\rightarrow Y$ with underlying map
\begin{align}
\label{eq:derived_counit_map_of_the_form}
  \Susp^r\Tot^\res\mathfrak{C}(Y)\longrightarrow FY
\end{align}
corresponding to the identity map
\begin{align}
\label{eq:derived_identity_map}
  \function{\id}{\Tot^\res\mathfrak{C}(Y)}{\Tot^\res\mathfrak{C}(Y)}
\end{align}
in $\Space_*$, via the adjunctions \cite[5.4]{Blomquist_Harper_integral_chains} and \eqref{eq:tensordot_adjunction_isomorphisms}. In more detail, the derived counit map is the derived $\Kt$-coalgebra map defined by the composite
\begin{align}
\label{eq:derived_counit_map}
  \Delta[-]\xrightarrow{(*)}
  &\Hombold_{\Space_*}\bigl(\Tot^\res\mathfrak{C}(Y),\mathfrak{C}(Y)\bigr)\\
  \notag
  \Iso
  &\Hombold_{\Space_*}\bigl(\Susp^r\Tot^\res\mathfrak{C}(Y),F\Kt^\bullet Y\bigr)
\end{align}
in $(\sSet)^{\Delta_\res}$, where $(*)$ corresponds to the map \eqref{eq:derived_identity_map}, via the adjunctions \cite[5.4]{Blomquist_Harper_integral_chains} and \eqref{eq:tensordot_adjunction_isomorphisms}.
\end{defn}

\begin{prop}
\label{prop:induced_map_on_mapping_spaces_built_from_Q}
Let $X,X'$ be pointed spaces.
There are natural morphisms of mapping spaces of the form
\begin{align*}
  \Susp^r\colon\thinspace\Hombold_{\Space_*}(X,X')
  \rarrow&\Hombold_\coAlgKt(\Susp^r X,\Susp^r X'),\\
  \Susp^r\colon\thinspace\Map_{\Space_*}(X,X')
  \rarrow&\Map_\coAlgKt(\Susp^r X,\Susp^r X'),
\end{align*}
in $\sSet$ and $\CGHaus$, respectively.
\end{prop}

\begin{prop}
There is an induced functor
\begin{align*}
  \function{\Susp^r}{\Ho(\Space_*)}{\Ho(\coAlgKt)}
\end{align*}
which on objects is the map $X\mapsto \Susp^r X$ and on morphisms is the map
\begin{align*}
  [X,X']\rarrow [\Susp^r X,\Susp^r X']_\Kt
\end{align*}
which sends $[f]$ to $[\Susp^r f]$.
\end{prop}

\begin{prop}
\label{prop:cosimplicial_resolutions_of_K_coalgebras_respect_adjunction_isos}
Let $X\in\Space_*$ and $Y\in\coAlgKt$. The adjunction isomorphisms associated to the $(\Susp^r,\Loop^r)$ adjunction induce well-defined isomorphisms
\begin{align*}
  \Hombold\bigl(\Susp^r X,F\Kt^\bullet Y\bigr)&\xrightarrow{\Iso}
  \Hombold\bigl(X,\Loopt^r\Kt^\bullet Y\bigr)
\end{align*}
of cosimplicial objects in $\sSet$, natural in $X,Y$.
\end{prop}

\begin{prop}
\label{prop:zigzag_of_weak_equivalences_tot_and_tq_completion}
If $X$ is a pointed space, then there is a zigzag of weak equivalences
\begin{align*}
  X^\wedge_{\Loopt^r\Susp^r}\wequiv
  \holim\nolimits_\Delta \mathfrak{C}(\Susp^r X)\wequiv
  \Tot^\res\mathfrak{C}(\Susp^r X)
\end{align*}
in $\Space_*$, natural with respect to all such $X$.
\end{prop}

\begin{defn}
A pointed space $X$ is \emph{$\Loopt^r\Susp^r$-complete} if the natural coaugmentation $X\wequiv X^\wedge_{\Loopt^r\Susp^r}$ is a weak equivalence.
\end{defn}

\begin{prop}
\label{prop:fundamental_adjunction_derived_version}
There are natural zigzags of weak equivalences of the form
\begin{align*}
  \Map_\coAlgKt(\Susp^r X,Y)\wequiv
  \Map_{\Space_*}(X,\holim\nolimits_\Delta \mathfrak{C}(Y))
\end{align*}
in $\CGHaus$; applying $\pi_0$ gives the natural isomorphism $[\Susp^r X,Y]_\Kt\Iso[X,\holim_\Delta \mathfrak{C}(Y)]$.
\end{prop}

\begin{proof}
There are natural zigzags of weak equivalences of the form
\begin{align*}
  \Hombold_{\Space_*}(X,\holim\nolimits_\Delta \mathfrak{C}(Y))
  &\wequiv\Hombold_{\Space_*}\bigl(X,\Tot^\res \mathfrak{C}(Y)\bigr)\\
  &\Iso\Tot^\res\Hombold_{\Space_*}\bigl(X,\Loopt^r\Kt^\bullet Y\bigr)\\
  &\Iso\Tot^\res\Hombold_{\Space_*}\bigl(\Susp^r X,F\Kt^\bullet Y\bigr)\\
  &\Iso\Hombold_\coAlgKt(\Susp^r X,Y)
\end{align*}
in $\sSet$; applying realization finishes the proof.
\end{proof}

The following shows that the suspension spaces functor in the fundamental adjunction is homotopically fully faithful on $\Loopt^r\Susp^r$-complete spaces; compare Hess \cite[5.5]{Hess} and Arone-Ching \cite[2.15]{Arone_Ching_classification}.

\begin{prop}
\label{prop:formal_adjunction_and_iso_argument}
Let $X,X'$ be pointed spaces. If $X'$ is $\Loopt^r\Susp^r$-complete and fibrant, then there is a natural zigzag
\begin{align*}
  \Susp^r\colon\thinspace\Map_{\Space_*}(X,X')\xrightarrow{\wequiv}
  \Map_\coAlgKt(\Susp^r X,\Susp^r X')
\end{align*}
of weak equivalences; applying $\pi_0$ gives the map $[f]\mapsto[\Susp^r f]$.
\end{prop}

\begin{proof}
This follows from the natural zigzags
\begin{align*}
  \Map_{\Space_*}(X,X')&\wequiv
  \Map_{\Space_*}(X,{X'}^\wedge_{\Loopt^r\Susp^r})\\
  &\wequiv\Map_{\Space_*}\bigl(X,\holim\nolimits_\Delta \mathfrak{C}(\Susp^r X')\bigr)
  \wequiv\Map_\coAlgKt(\Susp^r X,\Susp^r X')
\end{align*}
of weak equivalences.
\end{proof}

\bibliographystyle{plain}
\bibliography{IteratedSuspension}

\end{document}